\documentclass [12pt]{amsart} 

\usepackage{amsmath}
\usepackage{amssymb}
\usepackage{fullpage}
\usepackage{enumerate}
\usepackage{verbatim}

\newcommand{\R}{\mathbb{R}}
\newcommand{\N}{{\mathbb N}}
\newcommand{\Z}{{\mathbb Z}}
\newcommand{\T}{\mathbb{T}}

\newcommand{\ov}{\overline}
\newcommand{\spa}{\operatorname{span}}

\numberwithin{equation}{section}
\newtheorem{theorem}{Theorem}[section]
\newtheorem{corollary}[theorem]{Corollary}
\newtheorem{lemma}[theorem]{Lemma}
\newtheorem{proposition}[theorem]{Proposition}

\theoremstyle{definition}
\newtheorem{definition}{Definition}[section]
\newtheorem{remark}{Remark}[section]

\newcommand{\real}{\mathbb{R}}
\newcommand{\nat}{{\mathbb N}}
\newcommand{\inte}{{\mathbb Z}}

\newtheorem{problem}{Problem}
\newtheorem{factD}[theorem]{Fact}
\newtheorem{conjecture}{Conjecture}
\newcommand{\thm}[2]{\begin{theorem}\label{#1}#2\end{theorem}}
\newcommand{\cor}[2]{\begin{corollary}\label{#1}#2\end{corollary}}

\newcommand{\lem}[2]{\begin{lemma}\label{#1}#2\end{lemma}}

\newcommand{\defi}[2]{\begin{definition}\label{#1}#2\end{definition}}

\begin{document}

\title{Measurable  selector in Kadison's carpenter's theorem}

\author{Marcin Bownik}
\address{Department of Mathematics, University of Oregon, Eugene, OR 97403--1222, USA}
\email{mbownik@uoregon.edu}

\author{Marcin Szyszkowski}
\address{International School of Gda\'nsk,
ul. Sucha 29, 80-531 Gda\'nsk, Poland}
\email{fox@mat.ug.edu.pl}

\date{\today}



\thanks{The first author was supported in part by the NSF grant DMS-1665056.}

\begin{abstract} 
We show the existence of a measurable selector in Carpenter's Theorem due to Kadison \cite{k1, k2}. This solves a problem posed by Jasper and the first author in \cite{bj}. As an application we obtain a characterization of all possible spectral functions of shift-invariant subspaces of $L^2(\R^d)$ and Carpenter's Theorem for type I$_\infty$ von Neumann algebras.
\end{abstract}

\maketitle 

\section{Introduction}

Kadison \cite{k1, k2} gave a complete characterization of the diagonals of orthogonal projections on a separable infinite dimensional Hilbert space $H$. 

\begin{theorem}[Kadison]\label{Kadison} Let $(d_{i})_{i\in \N}$ be a sequence in $[0,1]$. Define
\[a=\sum_{d_{i} \le 1/2}d_{i} \quad\text{and}\quad b=\sum_{d_{i}> 1/2}(1-d_{i}).\]
There exists a projection $P$ with diagonal $(d_{i})_{i\in \N}$ if and only if one of the following holds:
\begin{itemize}
\item $a,b<\infty$ and  $a-b\in\Z$,
\item $a=\infty$ or $b=\infty$.
\end{itemize}
\end{theorem}

Kadison \cite{k1,k2} referred to the necessity part of Theorem \ref{Kadison} as the Pythagorean Theorem and the sufficiency as Carpenter's Theorem. It has been studied by a number of authors \cite{ar, amrs, mbjj, bj, kl}. Kadison's Theorem can be generalized to the setting of von Neumann algebras. A general version of the Schur-Horn problem asks for a characterization of possible diagonals of an operator based on its spectral data. In the setting of von Neumann algebras the notion of a diagonal of an operator is replaced by the conditional expectation onto a maximal abelian self-adjoint subalgebra (MASA).  The general Schur-Horn problem can then be formulated in the following way, see \cite{ks}.

\begin{problem}\label{psh}
Let $T$ be an operator in a von Neumann algebra $\mathfrak M$, an let $\mathcal A$ be a MASA in $\mathfrak M$ with corresponding conditional expectation $E_{\mathcal  A} : \mathfrak M \to \mathcal A$. Characterize the elements of the set
\begin{equation}\label{di0}
D_{\mathcal A}(T):=\{E_{\mathcal A}(U^*TU) : U \text{ is a unitary in }\mathfrak M\}.
\end{equation}
\end{problem}

This research area has been initiated by Arveson and Kadison \cite{ak, k2} who have asked for a characterization of $D_{\mathcal A}(T)$ when $T$ is a  projection, or more generally a self-adjoint operator, in a von Neumann factor of type II$_1$. Problem \ref{psh} was investigated by a number of authors \cite{am1, am2, am3, bhra, dfhs} and settled by Ravichandran \cite{mr, rav}. The same problem when $T$ is a normal operator was studied by Arveson \cite{a}, and by Kennedy and Skoufranis \cite{ks}, who have shown approximate Carpenter's Theorem for the closure of $D_{\mathcal A}(T)$ in von Neumann factors of type I${}_\infty$, II, and III. 
In the setting of von Neumann factor I$_\infty$, when $\mathfrak M=\mathcal B(H)$, Problem \ref{psh} was studied by Kaftal, Loreaux, and Weiss \cite{kw, lw} when $T$ is a positive compact operator and by Jasper and the first author when $T$ has finite spectrum \cite{bj2, jas}.

The goal of this paper is to show 
the existence of a measurable selector in Kadison's Carpenter's Theorem. This problem was explicitly posed by Jasper and the first author in \cite{bj}, although it was motivated by the earlier work of Rzeszotnik and the first author on spectral functions of shift-invariant spaces \cite{br, br2}.

\begin{definition}\label{mp}
Let $X$ be a measurable space and $H$ be a separable Hilbert space.
We say that $P: X \to \mathcal B( H)$ is a {\it measurable projection}, if:
\begin{enumerate}[(i)]
\item $P(x)$ is an orthogonal projection for all $x\in X$, and 
\item $P$ is weakly operator measurable, i.e., 
\[
x \mapsto \langle P(x) u,v \rangle \quad\text{is measurable for all }u,v\in H.
\]
Equivalently, all entries of $P(x)$ with respect to a fixed orthonormal basis $(e_i)_{i\in\N}$ of $H$ are  measurable.
\end{enumerate}
The {\it diagonal} of $P$ is a sequence of measurable functions $(\langle P(x)e_i, e_i \rangle)_{i\in \N}$.
\end{definition}

Our main result takes the following form.

\begin{theorem}\label{sane}
Let $X$ be a measurable space.
Let $f_i:X\to [0,1]$, $i\in\N$, be measurable functions. Define functions $a,b: X \to [0,\infty]$ by 
\begin{equation}\label{sane0}
a(x)=\sum_{i\in\N: f_i(x)\leq {1\over 2}} f_i(x) ~\quad\text{and}~\quad b(x)=\sum_{i\in\N: f_i(x)>{1\over 2}}1-f_i(x) \qquad\text{for } x\in X.
\end{equation}
Assume that for every $x\in X$, we have either:
\begin{enumerate}[(i)]
\item
$a(x)=\infty$ or $b(x)=\infty$, or
\item
$a(x), b(x)<\infty$ and $a(x)-b(x)\in\Z$.
\end{enumerate}
Then, there exists a measurable projection $P: X \to \mathcal B( H)$
with diagonal $(f_i)_{i\in\N}$.
\end{theorem}

Theorem \ref{sane} also yields a characterization of spectral functions of shift-invariant subspaces of $L^2(\R^n)$, which were introduced and studied by Rzeszotnik and the first author \cite{br, br2}. In fact, this was the main motivation for studying this problem. In addition, Theorem \ref{sane} solves Problem \ref{psh} for von Neumann algebras of type I$_\infty$ when $T$ is a projection. 

The proof of Theorem \ref{sane} splits into two natural cases: the nonsummable case (i) and summable case (ii), which are shown in Sections \ref{S2} and \ref{S3}, respectively. The nonsummable case of Theorem \ref{sane} is based on an algorithmic technique for finding a projection with all diagonal entries in $[0,1/2]$ except possibly one term. This technique was introduced by Jasper and the first author in \cite[Section 4]{bj}. It is related to the spectral tetris construction of tight frames introduced by Casazza, Fickus, Mixon, Wang, and Zhou in \cite{cfmwz}. The proof of the case (i) also relies heavily on techniques of measurable permutations.

The proof of the summable case (ii) employs a measurable variant of the finite dimensional Schur-Horn theorem which was shown by Benac, Massey, and Stojanoff in \cite{bms}. The key role is played by a decoupling procedure that splits a desired diagonal sequence into three parts modifying at most one entry in each group. The resulting sequences correspond either to infinite dimensional rank one projections or finite dimensional projections. Then, the measurable variant of the Schur-Horn theorem enables us to recover a projection with the original diagonal.

Throughout the paper $X$ denotes a measurable space. For a function $f:X\to\real$ and $a\in\real$, the set $\{f=a\}$ stands for $\{x\in X:f(x)=a\}=f^{-1}(a)$; similarly, we define $\{f<a\}, \{f\leq a\}$, etc.

\section{Nonsummable case} \label{S2}

In this section we prove the nonsummable case of Theorem \ref{sane}. 

\thm{main}{Let $f_i:X\to [0,1]$, $i\in \N$, be measurable functions. Define functions $a,b:X \to [0,\infty]$ by \eqref{sane0}. If for every $x\in X$, $a(x)=\infty$ or $b(x)=\infty$, then there is a measurable projection $P: X \to \mathcal B( H)$ with diagonal $(f_i)_{i\in \N}$.
}

We will repeatedly use the following two elementary lemmas.
\lem{sklejanie}{Let $X_i\subset X$, $i\in \N$, be measurable, $Z$ be a topological space, and $g_i:X_i\to Z$ be measurable functions.
If $X_i$, $i\in\N$, are pairwise disjoint, then the function $g=\bigcup g_i: \bigcup X_i \to Z$ is measurable.}

\lem{index}{Let $f_i: X \to \R$, $i\in \N$, and $h:X\to\nat$ be measurable. Define a function 
\[
f_h: X \to \R, \qquad f_h(x)=f_{h(x)}(x), \ x\in X.
\]
Then, $f_h$ is measurable.}

\begin{proof} For any $k\in\N$, the set $\{h=k\}$ is measurable and on this set the function $f_h=f_k$ is measurable. \end{proof}

We first prove Theorem \ref{main} in a very special case given by Theorem \ref{mainspec}. This is precisely the setting of the algorithm in \cite[Section 4]{bj}. To do this we need to introduce some auxiliary functions. 

\begin{definition}\label{in}
Let $f_i:X\to [0,1]$, $i\in \N$, be measurable functions and $N\in \N \cup\{\infty\}$ be such that 
\begin{equation}\label{in0}
\sum_i f_i(x) = N \qquad\text{for all }x\in X.
\end{equation}
For any $n<N$, $n\in \N$, we define
$minS(n):X\to\nat$ by 
\[
minS(n)(x)=\min\{i:S_i(x)\geq n\},\qquad\text{where }
S_i=\sum_{j=1}^i f_j.
\]
To emphasize the dependence on $(f_i)_{i\in \N}$, we shall use the notation $minS((f_i),n)$.
\end{definition}

The assumption \eqref{in} guarantees that a function $minS(n)$ is defined on the whole $X$ when $n<N$. In the next three results we will assume that the sequence of functions $(f_i)_{i\in \N}$ is as in Definition \ref{in}.

\lem{minS}{For any $n<N$, a function $minS(n)$ is measurable. }

\begin{proof} For all $k\in \N$ a function $S_k$ is measurable. Fix $n$ and $k$. The set
\[
\{minS(n)=k\}=\{S_k\geq n\}\cap \{S_{k-1}<n\}
\]
 is measurable. Function $minS(n)$ is constant
(equal to $k$) on $\{minS(n)=k\}$. By Lemma \ref{sklejanie} we are done. 
\end{proof}

As a corollary of Lemmas  \ref{index} and \ref{minS} we have
\cor{fminS(n)}{For any $n<N$, functions  
\[
f_{minS(n)} \qquad\text{and}\qquad \sum_{i=1}^{minS(n)} f_i~=~S_{minS(n)}
\]
are measurable.}

\thm{mainspec}{Let $(f_i)_{i\in \N}$ be a sequence of measurable functions and let $N \in \N \cup \{\infty\}$. Suppose that for all $x\in X$,
\begin{enumerate}[(i)]
\item $\sum_i f_i(x)=N$,
\item
$f_i(x)\in [0,{1\over 2}]$ for $i>1$, and 
\item
$f_{minS(n)-1}(x) \geq f_{minS(n)}(x)$ for all $1\le n<N$.
\end{enumerate}
Then, there is a measurable projection $P: X \to \mathcal B( H)$ such that $P(x)$ has diagonal $(f_i(x))_{i\in\N}$ for all $x\in X$. }

\begin{proof} The proof is a repetition of the algorithm in \cite [Theorem 4.3]{bj}. Consequently, we only need to verify
that this construction yields a measurable projection as in Definition \ref{mp}.

For any $n<N$, define a function $\sigma_n: X \to [0,1]$ by
 $$\sigma_n(x)=n - S_{minS(n)-2}(x) =n-\sum_{i=1}^{minS(n)(x)-2}f_i(x).$$
Function $\sigma_n$ is measurable as a consequence of Corollary \ref{fminS(n)}.
 Note that 
 \[
 \sigma_n\leq f_{minS(n)-1}+f_{minS(n)}.
 \]
By the minimality of $minS(n)$ we have
\[
\sigma_n = n- S_{minS(n)-1}+f_{minS(n)-1}>f_{minS(n)-1}.
\]
Hence, by (iii) we have
\[
2\sigma_n>f_{minS(n)-1}+f_{minS(n)}.
\]
We will use the the following elementary lemma from \cite{bj}.

 \lem{numbera}{{\cite[Lemma 4.1]{bj}} Let $d_1,~d_2,~\sigma\in [0,1]$, $\max \{d_1,d_2\}\leq\sigma\leq d_1+d_2$ and $2\sigma >d_1+d_2$.
 The number $$ a={{\sigma (\sigma -d_2)}\over{2\sigma-d_1-d_2}}$$
 satisfies:
 $a,~\sigma -a,~d_1-a,~d_2-\sigma +a \in [0,1]$ and \newline
\begin{equation}\label{numbera1}
a(d_1-a)=(\sigma-a)(d_2-\sigma +a).
\end{equation} }
 Substituting values of $f_{minS(n)-1}=d_1$ and $f_{minS(n)}=d_2$ and $\sigma_n$ as $\sigma$, Corollary \ref{fminS(n)} yields
the measurability of a function $a_n: X\to [0,1]$ given by
 $$a_n={{\sigma_n(\sigma_n-f_{minS(n)})}\over{2\sigma_n-f_{minS(n)-1}-f_{minS(n)}}}.$$
In particular, functions $
\sigma_n-a_n$, $f_{minS(n)-1}-a_n$, $f_{minS(n)}-\sigma_n+a_n$ are measurable, and so are functions $\sqrt{\sigma_n-a_n}$, $\sqrt{f_{minS(n)-1}-a_n}$, $\sqrt{f_{minS(n)}-\sigma_n+a_n}: X \to [0,1]$.

\smallskip
Let $(e_i)_{i\in \N}$ be an orthonormal basis of the space $H$. We define a sequence of vectors $(v_n)_{n<N}$ by
\[
\begin{aligned}
v_1 &=\sum_{i=1}^{minS(1)-2}\sqrt{f_i} e_i + \sqrt{a_1} e_{minS(1)-1}
-\sqrt{\sigma_1-a_1} e_{minS(1)},
\\
v_n &=\sqrt{f_{minS(n-1)-1}-a_{n-1}} e_{minS(n-1)-1} + \sqrt{f_{minS(n-1)}-\sigma_{n-1}+a_{n-1}} e_{minS(n-1)}
\\
 &\quad +\sum_{i=minS(n-1)+1}^{minS(n)-2}\sqrt{f_i} e_i +\sqrt{a_n} e_{minS(n)-1}-\sqrt{\sigma_n-a_n} e_{minS(n)}, \qquad 1<n<N.
\end{aligned}
\]
Vectors $v_n$ are well-defined since we have $minS(n-1) +2 \le minS(n)$ due to the assumption (ii).
In the case when $N<\infty$, we also define the ultimate vector
\[
\begin{aligned}
v_N =\sqrt{f_{minS(N-1)-1}-a_{N-1}} e_{minS(N-1)-1} & + \sqrt{f_{minS(N-1)}-\sigma_{N-1}+a_{N-1}} e_{minS(N-1)}
\\
&+ \sum_{i=minS(N-1)+1}^{\infty}\sqrt{f_i} e_i.
\end{aligned}
\]

A direct calculation using \eqref{numbera1} shows that $(v_n)_{n\in \N}$ forms an orthonormal set of vectors in $H$.
We refer the reader to \cite[Theorem 4.3]{bj} for the proof in the case $N=\infty$. The same is true for $(v_n)_{n=1}^N$ in the case $N<\infty$. Indeed, this follows by the fact that the support of the ultimate vector $v_N$ overlaps only with the preceding vector $v_{N-1}$.

We can represent vectors $(v_n)_{n=1}^N$ as rows of an $N\times \infty$ matrix $V$\\
\[
\begin{bmatrix}
\sqrt{f_1}&\ldots&\sqrt{a_1}& -\sqrt{\sigma_1-a_1}&  &  \\
 &  & \sqrt{f_\triangle -a_1} &\sqrt{f_\spadesuit - \sigma_1+a_1}&\sqrt{f_\heartsuit} & \ldots\hspace{-2pt}&\sqrt{a_2}& -\sqrt{\sigma_2-a_2} & & \\
&&&&&& \sqrt{f_\clubsuit -a_2}& \sqrt{f_\diamondsuit -\sigma_2+a_2} &\sqrt{f_\nabla} & \ldots 
\\
&&&&&&&&& \ldots
 \end{bmatrix} 
 \]
where $f_\triangle=f_{minS(1)-1},~f_\spadesuit =f_{minS(1)},~f_\heartsuit =f_{minS(1)+1},~f_\clubsuit =f_{minS(2)-1},
~f_\diamondsuit =f_{minS(2)},~f_\nabla =f_{minS(2)+1}$, and the empty spaces are zeros.
That is, row $n$ of the matrix $V$ represents coefficients of the vector $v_n$ with respect to the basis $(e_i)_{i\in \N}$,
\[
v_n=\sum_{k\in\N} V_{n,k} e_k.
\]

{\bf Claim.} {For any $n$ and $k$, a function $V_{n,k}:X\to [0,1]$ is measurable.}
\begin{proof}
Observe that
$$ V_{n,k}= 
\begin{cases}
0 & k<minS(n-1)-1, \\
\sqrt{f_{minS(n-1)-1}-a_{n-1}} & k=minS(n-1)-1, \\
\sqrt{f_{minS(n-1)}-\sigma_{n-1}+a_{n-1}} & k=minS(n-1) ,\\
\sqrt{f_k} & minS(n-1)<k<minS(n)-1, \\
\sqrt{a_n} & k=minS(n)-1, \\
\sqrt{\sigma_n-a_n} & k=minS(n), \\
0 & k> minS(n). \end{cases}
$$
In the above we set $minS(0)=a_0=\sigma_0=0$. If $N<\infty$, then we also set $minS(N)=\infty$, consequently, the last three cases are vacuous.
The sets defined by the above cases, such as $\{k=minS(n-1)\}$,
are measurable. Therefore, the measurability of $minS(i), f_i, a_i,\sigma_i$ together with Lemmas \ref{sklejanie} and \ref{index} yields the claim.
\end{proof}

The sought measurable projection $P$ is given by the formula 
$$Pv=\sum_{n=1}^N \langle v,v_n \rangle v_n, \qquad v\in H. $$
Indeed, a computation as in \cite[Theorem 4.3]{bj} shows that diagonal of $P$ is as desired, i.e., $\langle Pe_i,e_i \rangle =f_i$. Hence, it remains to show that the mapping $x \mapsto P(x)\in \mathcal B(H)$ is weakly measurable. For any $k\in \N$,
\[
P e_k = \sum_{n=1}^N \langle e_k,v_n \rangle v_n = \sum_{n=1}^N V_{n,k} v_n = \sum_{n=1}^{\min((k+3)/2,N)} V_{n,k} v_n.
\]
The limitation in the last sum follows from the fact that $V_{n,k}=0$ for $k<minS(n-1)-1$ and the fact that $minS(n) \ge 2n$ for all $n<N$. Since a product of a measurable vector-valued function by a measurable scalar-valued function is measurable, the above claim implies that the mapping $x\mapsto P(x)e_k$ is measurable as well. Since $k\in \N$ is arbitrary, $P$ is a measurable projection.
\end{proof}

In order to weaken the assumptions of Theorem \ref{mainspec}, we need to introduce the concept of a measurable permutation.

\begin{definition}\label{defmu}
We say that $\pi: X \times \N \to \N$ is a {\it measurable permutation}, if for all $n\in\N$, $\pi(\cdot,n)$ is measurable and for all $x\in X$, $\pi(x,\cdot)$ is a permutation of $\N$. The inverse permutation $\pi^{-1}: X \times \N \to \N$ is given by 
\[
\pi^{-1}( x,n) = (\pi(x,\cdot))^{-1}(n), \qquad (x,n) \in X \times \N.
\]
\end{definition}

Note that if $\pi$ is measurable, then so is the inverse permutation. Indeed, for any $k\in\N$,
\[
\{(x,n): \pi^{-1}(x,n)=k\} = \{ (x,n) : \pi(x,k)=n\} 
\]
is a measurable set. For any $n\in \N$, we let $\pi(n): X \to \N$ to denote a function given by $\pi(n)(x)=\pi(x,n)$, $x\in X$.

We need the following two useful lemmas.

\begin{lemma}\label{permut} Let $\pi: X \times \N \to \N$ be a measurable permutation.
Suppose that $P: X \to \mathcal B(H)$ is a measurable projection with diagonal $(f_i)_{i\in \N}$. Then, there exists a measurable projection $\tilde P: X \to \mathcal B(H)$ with diagonal $(f_{\pi(i)})_{i\in \N}$.
\end{lemma}

\begin{proof}
A measurable permutation $\pi$ defines a permutation mapping $U: X \to \mathcal B(H)$ given by 
\[
U(x)e_i = e_{\pi(x,i)} \qquad i \in \N, x\in X.
\]
It is immediate that $x\mapsto U(x)$ is weakly measurable and $U(x)$ is a unitary operator. Define $\tilde P: X \to \mathcal B(H)$ by 
\[
\tilde P(x) = (U(x))^* \circ P(x) \circ U(x) \qquad\text{for }x\in X.
\]
A straightforward argument shows that $\tilde P$ is a measurable projection with diagonal $(f_{\pi(i)})_{i\in \N}$. Indeed, for any $(x,i) \in X \times \N$,
\[
\langle \tilde P(x) e_i, e_i \rangle = 
\langle  P(x) U(x) e_i, U(x) e_i \rangle 
= \langle P(x) e_{\pi(x,i)}, e_{\pi(x,i)} \rangle = f_{\pi(i)}(x).
\]
This proves the lemma.
\end{proof}

\lem{ordering}{Let $f_i:X\to [0,1]$, $i\in [k]:=\{1,\ldots,k\}$,  be a finite sequence of measurable functions. Then, there exists a measurable permutation $\pi: X \times [k] \to [k]$ such that functions $g_i=f_{\pi(i)}$, $i\in [k]$, are measurable and in (weakly) decreasing order. 
}

\begin{proof} Let $a_1,..,\breve{a}_i,..$ denotes the sequence $a_1,.,.a_{i-1},a_{i+1},..$ with $a_i$ omitted.
Define functions $g_i: X \to [0,1]$ by  
\[
\begin{aligned}
g_1&=\max(f_1,f_2,...,f_k), \\
g_2&=\min_{i\leq k} \max(f_1,..,\breve{f_i},..,f_k), 
\\
g_3 & =\min_{i<j\leq k} \max(f_1,...,\breve{f_i},..,\breve{f_j},..,f_k),
\end{aligned}
\]
and so on. For a fixed $x\in X$, define 
\[
\begin{aligned}
\pi(1) &= \min \{ i: f_i=g_1\}, 
\\
\pi(2) &= \min\{ i: f_i=g_2 \text{ and } i \ne \pi(1) \},
\\
\pi(3) &= \min\{ i: f_i=g_3 \text{ and } i \ne \pi(1), \pi(2) \},
\end{aligned}
\]
and so on. Then, one can show by induction that functions $g_i$ and $\pi(\cdot,i)$ are measurable for every $i=1,\ldots,k$.
Finally, the identity $g_i=f_{\pi(i)}$ follows by the above definition.
\end{proof}




We now prove Theorem \ref{mainspec} under weaker assumptions.

\thm{mainspec2}{Suppose measurable functions $f_i:X\to [0,1]$, $i\in \N$, satisfy for all $x\in X$,
\begin{enumerate}[(i)]
\item $\sum_i f_i(x) =\infty$ or  $\sum_i f_i(x)\in\N$, and 
\item $f_i(x)\in [0,{1\over 2}]$ for all $i>1$.
\end{enumerate}
Then, there is a measurable projection $P: X \to \mathcal B( H)$ with diagonal $(f_i)_{i\in \N}$.}

\begin{proof} Lemma \ref{sklejanie} implies that, without loss of generality,  we can assume that there exists $N \in \N \cup \{\infty\}$ such that \eqref{in0} holds.
The idea of the proof is to find a measurable permutation $\pi:X \times \nat\to\nat$ if $N=\infty$, or $\pi:X \times [N] \to [N]$ if $N<\infty$, such that the permuted sequence $g_i=f_{\pi(i)}$ satisfies
\begin{equation}
\label{ord1}
g_{minS((g_i),n)-1}\geq g_{minS((g_i),n)}
\qquad\text{for all } 1\le n<N.
\end{equation}
That is, $(g_i)_{i\in\N}$ satisfies the assumption (iii) from Theorem \ref{mainspec}. For simplicity we shall assume that $N=\infty$; the case $N<\infty$ follows by obvious modifications.

To achieve this we merely follow the proof in \cite[Lemma 4.2]{bj}.
For every $x\in X$, we divide the sequence $f_i$ into blocks corresponding to intervals defined by $minS(n)$ given by
\begin{equation}\label{ord2}
I_n=\{i\in \N: minS(n-1)<i\leq minS(n)\}, \qquad n \in \N,
\end{equation}
with the convention that $minS(0)=0$.
On every such interval we order $(f_i)_{i\in I_n}$ in a decreasing order using Lemma \ref{ordering}. More precisely, for every $n\in \N$, there are countable many choices for an interval $I_n$. Restricting to $x\in X$ such that $I_n =[a,b]:=\{a,\ldots, b\}$ for some fixed $a<b\in \N$, we apply Lemma \ref{index} to $(f_i)_{i\in [a,b]}$ to get a local measurable permutation of $I_n$. We combine these permutations into one global permutation $\pi:X \times \nat\to\nat$ that sorts every block of functions $(f_i)_{i\in I_n}$ in a decreasing order. By Lemma \ref{index}, $\pi$ is a measurable permutation. Indeed, for a fixed $i\in \N$, measurable spaces $X$ splits into at most countably subsets indexed by triplets $(n,a,b) \in \N^3$ such that $i\in I_n=[a,b]$. Clearly, $\pi(\cdot,i)$ is measurable on each such subset and hence on $X$.

Define $g_i=f_{\pi(i)}$ for $i\in \N$. By Lemma \ref{index},  functions $g_i$ are measurable. We claim that that the sequence $(g_i)$ satisfies \eqref{ord1}. To see this we must consider functions $minS(n)$ corresponding to $(g_i)$.
The values of $minS(n)$ on the sequence $(g_i)$ may differ from  analogous values on $(f_i)$ by at most $2$. Indeed, by formula (4.7) in \cite[Lemma 4.2]{bj} we have
\begin{equation}\label{ord3}
minS((f_i),n-1)+2\leq minS((g_i),n)\leq minS((f_i),n), \qquad n\in\N.
\end{equation}
Since $(g_i)_{i\in I_n}$ is in decreasing order, \eqref{ord2} and \eqref{ord3} yields \eqref{ord1}.

By Theorem \ref{mainspec} there exists a measurable projection $\tilde P$ with diagonal $(g_i)_{i\in \N}=(f_{\pi(i)})_{i\in \N}$. Applying Lemma \ref{permut} for the inverse permutation $\pi^{-1}$ yields a measurable projection $P$ with diagonal $(g_{\pi^{-1}(i)})_{i\in \N} = (f_i)_{i\in \N}$.
This completes the proof of Theorem \ref{mainspec2}. 
\end{proof}

Our goal now is to prove Theorem \ref{mainspec2} without the assumption (ii).
To this end, we need some auxiliary functions. 

\defi{pos}{For a sequence of measurable functions $f_i:X\to [0,1]$, $i\in \N$, we
define functions $Pos(n):X\to\nat$ and $pos(n):X\to\nat$ as follows. Let $Pos(n)(x)=k$ if
$f_k(x)$ is the $n$-th number in the sequence $(f_i(x))$ that is greater than ${1\over 2}$. Likewise, we let $pos(n)(x)=k$ if $f_k(x)$ is the $n$-th number in the sequence $(f_i(x))$ that is $\leq{1\over 2}$.}

In order for $Pos(n)$ and $pos(n)$ to be defined on the whole $X$, we must assume that there are at least $n$ indices $i\in \N$ such that $f_i(x)> \frac12$ and $f_i(x) \le \frac12$, respectively.

\lem{position}{Functions $Pos(n)$ and $pos(n)$ are measurable for all $n$.}
\begin{proof} Fix $n$ and $k$. Sets $\{Pos(n)=k\}$ and $\{pos(n)=k\}$ are measurable since 
\[
\begin{aligned}
\{Pos(n)=k\} & =\bigcup_{1\leq k_1<...<k_n=k}
\bigcap_{i=1}^k \left(  \{ f_{k_i}>{\tfrac12}\} \cap\bigcap_{j\neq k_i,j<k} 
\{ f_j\leq{ \tfrac12}\}\right),
\\
\{pos(n)=k\} & =\bigcup_{1\leq k_1<...<k_n=k} 
\bigcap_{i=1}^k  \left(  \{f_{k_i}\leq{ \tfrac12}\} \cap\bigcap_{j\neq k_i,j<k}\{ f_j>{ \tfrac 12}\} \right).
\end{aligned}
\]
\end{proof}

As a corollary of Lemmas \ref{index} and \ref{position}, we obtain that $f_{Pos(n)}$ and $f_{pos(n)}$ are measurable functions. The following result is the next step toward proving Theorem \ref{main}.

\thm{mainspec3}{Suppose that measurable functions $f_i: X \to [0,1]$, $i\in \N$, satisfy for all $x\in X$, 
\begin{enumerate}[(i)]
\item
there are infinitely many $f_i(x) \le {1\over 2}$,
\item
there are infinitely many $f_i(x)>{1\over 2}$, and
\item 
$\sum_{i}f_{pos(i)}(x) =\infty$.
\end{enumerate}
Then, there is a measurable projection $P: X \to \mathcal B( H)$ with diagonal $(f_i)_{i\in \N}$.}

\begin{proof} To simplify the notation, we let $f^i=f_{pos(i)}$. Of course, $f^i: X \to [0,1/2]$, $i\in \N$ is a sequence of measurable functions. Let $2\N -1$ denote the set of odd numbers. 
We will use a partition of $\nat$ into sets $A^m$, given by 
\[
A^m=\{2^{k-1}m:k\in\nat\}, \qquad m\in 2\N-1.
\]
Next, we decompose $H$ as an orthogonal sum of subspaces 
\[
H^m = \ov{\spa} \{ e_i: i\in A^m\}, \qquad m\in 2\N-1.
\]
We also split the sequence $(f_i)_{i\in \N}$ into countably many subsequences $(g^m_i)_{i\in \N}$, where $m$ is odd, by the following procedure.

Let $minS(n)$, $n\in \N$, be a function as in Definition \ref{in} corresponding to the sequence $(f^i)_{i\in\N}$ rather than $(f_i)_{i\in\N}$. Let $I_n$, $n\in\N$, be an interval of $\N$ defined by \eqref{ord2}.  Recall that intervals $I_n$, $n\in\N$, form a partition of $\N$, which depends on a choice of $x\in X$. For $i\in\N$, define a function $n_i: X \to \N$ by
\[
n_i = \max\{ n\in \N: minS(n-1)<i \}.
\]
Equivalently, $n_i \in \N$ is a unique number such that $i \in I_{n_i}$. Each function $n_i$ is measurable since 
\[
\{n_i=n\} = \{minS(n-1)<i\} \cap \{minS(n) \ge i \}.
\]
Define a mapping $\pi: X \times (2\N-1) \times \N \to \N $ such that $\pi(x,m,i)$ is the $i$-th element of the (infinite) set
\begin{equation}\label{set}
\bigcup_{a\in A^m} I_{n_a} = \bigcup_{k\in \N} I_{n_{2^{k-1}m}}.
\end{equation}

{\bf Claim.} For every $(m,i) \in (2\N-1) \times \N$, the mapping $\pi(\cdot,m,i)$ is measurable. 
\begin{proof}
Since each interval $I_n$ is non-empty, the $i$-th element of the set \eqref{set} belongs to finitely many intervals $I_{n_{m2^{k-1}}}$, where $1\le k\le i$. Observe that there are only countably many choices for such intervals. Hence, we can split $X$ into a countable collection of measurable subsets on which
\[
I_{n_m} = [a_1,b_1], \ 
I_{n_{2m}} = [a_2,b_2],
\ldots,
I_{n_{2^{i-1}m}} = [a_i,b_i],
\]
for some choice of $a_1 \le b_1<a_2\le b_2< \ldots< a_i \le b_i \in \N$. Consequently, the function $\pi(m,i):=\pi(\cdot,m,i)$ takes a constant value on such measurable subsets.
By Lemma \ref{sklejanie}, $\pi(m,i)$ is measurable on $X$.
\end{proof}

Consequently, $\pi: X \times (2\N-1) \times \N \to \N $ is a measurable permutation after identifying $(2\N-1) \times \N$ with $\N$ using the mapping $(i,m) \mapsto 2^{i-1}m$.
For $m$ odd, define sequence $(g^m_i)_{i\in \N}$ as
\begin{equation}\label{set2}
g^m_i = \begin{cases}
f_{Pos({{m+1}\over 2})} & i=1,
\\
f^{\pi(m,i-1)} & i>1.
\end{cases}. 
\end{equation}
In particular,
\[
\begin{aligned}
(g^1_i) &=( f_{Pos(1)},f^1,\ldots,f^{minS(2)},
f^{minS(3)+1},\ldots,f^{minS(4)},
f^{minS(7)+1}, \ldots, f^{minS(8)},\ldots),
\\
(g^3_i)&=(f_{Pos(2)},f^{minS(2)+1},\ldots,f^{minS(3)},f^{minS(5)+1},\ldots,f^{minS(6)},f^{minS(11)+1},\ldots),
\\
(g^5_i)&=(f_{Pos(3)},f^{minS(4)+1},\ldots,f^{minS(5)},f^{minS(9)+1},\ldots,f^{minS(10)},f^{minS(19)+1},\ldots),
\end{aligned}
\]
and so on. By Lemma \ref{index} and the above claim we deduce that all functions $g^m_i$ are measurable. Since 
\[
\sum_{i=minS(n-1)+1}^{minS(n+1)}f^i\geq {1\over 2}\qquad\text{for } n\in \N,
\]
we have that for every odd $m$, $\sum_{i=1}^\infty g^m_i=\infty$ and
$g^m_i\leq{1\over 2}$ for $i>1$. Hence, by Theorem \ref{mainspec2} there is a measurable projection $Q^m$ on the space $H^m$ with diagonal 
$(g^m_i)_{i\in \N}$. Therefore,
\[
Q=\bigoplus_{m\in 2\N -1} Q^m \in \mathcal B\bigg(\bigoplus_{m\in 2\N -1} H^m \bigg) = \mathcal B(H)
\]
is a projection with diagonal $(g^m_i)_{i\in \N}^{m\in 2\N-1}$. 

We want, however, a projection with diagonal $(f_i)_{i\in \N}$. This is a consequence of the fact that the sequence $(g^m_i)_{i\in \N}^{m\in 2\N-1}$ is obtained from $(f_i)_{i\in \N}$ using a measurable permutation. More precisely, let $\tilde \pi:X\times\nat\to\nat$ be a measurable permutation defined for $x\in X$ by 
by 
\[
\tilde \pi(x,2^{i-1}m) = \begin{cases} 
Pos({{m+1}\over 2})(x) & i=1, m \in 2\N-1,
\\
pos(\pi(x,m,i-1))(x) & i>1, m\in 2\N-1.
\end{cases}. 
\]
Indeed, one can check that for every $x\in X$, $\pi(x,\cdot)$ is a permutation of $\N$. The measurability follows from the claim and Lemma \ref{index}. Furthermore, by \eqref{set2} we have $g^m_i = f_{\tilde \pi(2^{i-1}m)}$. Therefore, Lemma \ref{permut} yields a measurable permutation with diagonal $(f_i)_{i\in \N}$. \end{proof} 

We also need the following slight variant of Theorem \ref{mainspec3}.

\thm{mainspec4}{Let $k\in \N$. Suppose that measurable functions $f_i: X \to [0,1]$, $i\in \N$, satisfy for all $x\in X$, 
\begin{enumerate}[(i)]
\item
there are infinitely many $f_i(x) \le {1\over 2}$,
\item
there are exactly $k$ many $f_i(x)>{1\over 2}$, and
\item 
$\sum_{i}f_{pos(i)}(x) =\infty$.
\end{enumerate}
Then, there is a measurable projection $P: X \to \mathcal B( H)$ with diagonal $(f_i)_{i\in \N}$.
}

\begin{proof} The proof is a simple adaption of the proof of Theorem \ref{mainspec3} to the finite case. The main 
difference is that we split the sequence $(f_i)_{i\in \N}$ into $k$ finitely many subsequences $(g^m_i)_{i\in \N}$, where $m$ is 
the remainder of division by $k$. Hence,
\[
\begin{aligned} 
(g^1_i) &=(f_{Pos(1)},f_{pos(1)},f_{pos(k+1)},f_{pos(2k+1)}, f_{pos(3k+1)},\ldots ),
\\ 
(g^2_i)&=(f_{Pos(2)},f_{pos(2)}, f_{pos(k+2)}, f_{pos(2k+2)}, f_{pos(3k+2)},\ldots),
\\ 
&\ \ \vdots 
\\
(g^k_i)&=( f_{Pos(k)}, f_{pos(k)}, f_{pos(2k)}, f_{pos(3k)},f_{pos(4k)}, \ldots).
\end{aligned}
\] 
Similarly we decompose $H$ as an orthogonal sum of $k$ subspace $H^1,\ldots,H^k$. Then, Theorem \ref{mainspec2} yields a measurable projection with diagonal $(g^m_i)_{i\in \N}$ for every $m=1, \ldots,k$. Then, an application of Lemma \ref{permut} yields the required measurable projection. We leave the details to the reader.
\end{proof} 

Combining Theorems \ref{mainspec3} and \ref{mainspec4} yields the following corollary

\cor{mainspec5}{ Suppose that measurable functions $f_i: X \to [0,1]$, $i\in \N$, satisfy for all $x\in X$, 
\begin{enumerate}[(i)]
\item
there are infinitely many $f_i(x) \le {1\over 2}$, and
\item
$\sum_{i}f_{pos(i)}(x) =\infty$.
\end{enumerate}
Then, there is a measurable projection $P: X \to \mathcal B( H)$ with diagonal $(f_i)_{i\in \N}$.
}

\begin{proof} For fixed $k\in\N \cup \{0\}$, the set $X_k$ of $x\in X$ such that the set $\{i\in \N: f_i(x)> {1\over 2}\}$ has $k$ elements is measurable. Indeed,
\[
X_k=\bigcup_{1\leq i_1<...<i_k}  \bigcap_{j=1}^k \bigg( \{f_{i_j}>{\tfrac 12} \} 
\cap \bigcap_{n\ne i_j} \{f_n\leq{ \tfrac12}\}\bigg).
\]
In particular, the set $X_\infty=\{x\in X: \text{there are infinitely many } f_i(x)>{1\over 2}\}$ 
is measurable as $X_\infty=X\setminus\bigcup_k X_k$. 

By Theorems \ref{mainspec2}, \ref{mainspec3}, and \ref{mainspec4} there exists measurable projections $P_0$, $P_k$, $k\in\N$, and $P_\infty$ defined 
on sets $X_0$, $X_k$, and $X_\infty$, respectively. A projection $P=\bigcup_{k\in \N \cup \{0\}} P_k \cup P_\infty$ is defined on the entire $X$ 
and is measurable by Lemma \ref{sklejanie}. This completes the proof of Theorem \ref{mainspec5}.
\end{proof} 

We are now ready to complete the proof of Theorem \ref{main}.

\begin{proof} 
Let $f^i=f_{pos(i)}$, see Definition \ref{pos}. Let
\[ S_\infty=\{x \in X:\sum_i f^i(x)=\infty\} , \qquad 
S_{<\infty}=\{x \in X:\sum_i f^i(x)<\infty\}
\]
The sets $S_\infty$ and $S_{<\infty}$ are measurable. Applying Corollary \ref{mainspec5} on the set 
\[
S_\infty = \{x\in X: a(x)=\infty\}
\]
yields a measurable projection $P: S_\infty \to \mathcal B(H)$ with diagonal $(f_i)_{i\in \N}$.
On the set $\{x\in X: a(x)<\infty\}=S_{<\infty}$ we must have $b(x)=\infty$. By the previous case there is a measurable projection $P^\prime$ on $S_{<\infty}$ with diagonal $(1-f_i)_{i\in \N}$. Hence, $P=\mathbf{I}-P^\prime$ is a measurable projection on $S_{<\infty}$ with diagonal $(f_i)_{i\in \N}$. Applying Lemma \ref{sklejanie} for $X=S_\infty \cup S_{<\infty}$ yields the desired measurable projection.
\end{proof}

\section{Summable case}\label{S3}

The aim of this section is to prove the summable counterpart of Theorem \ref{main}. 

\thm{mainfin}{Let $f_i:X\to [0,1]$, $i\in \N$, be measurable functions. Define for every $x\in X$ 
$$ a(x)=\sum_{f_i(x)\leq {1\over 2}} f_i(x),
\qquad
b(x)=\sum_{f_i(x)>{1\over 2}} 1-f_i(x).$$ 
Assume for all $x\in X,~ a(x), b(x)<\infty$ and $a(x)-b(x)\in\inte$. \newline 
Then, there is a measurable projection $P: X \to \mathcal B( H)$ with diagonal $(f_i)_{i\in \N}$.
} 

We start from the easiest rank one case.

\lem{sum1}{Assume $\sum_i f_i(x)=1$ for all $x\in X$. Then, there is a measurable projection $P: X \to \mathcal B( H)$ with diagonal $(f_i)_{i\in \N}$.} 

\begin{proof} Fix an orthonormal basis $(e_i)$ of $H$ and define a vector-valued function $v_0=\sum_i \sqrt{f_i} \cdot e_i$. 
A projection $P$ onto one dimensional space given by 
$P v=\langle v_0,v\rangle v_0$ is the desired projection. Indeed, 
$Pe_i= \langle v_0,e_i \rangle v_0= \langle \sum_j\sqrt{f_j}e_j, e_i \rangle v_0=\sqrt{f_i}\cdot v_0$ and the entries of $P$ are 
$$\langle Pe_j,e_i\rangle =\langle\sqrt{f_j}\cdot v_0,e_i\rangle=\sqrt{f_j}\langle v_0,e_i\rangle=\sqrt{f_j}\cdot\sqrt{f_i}~.$$ 
\end{proof}

We will need a measurable variant of the Schur-Horn theorem \cite{horn, schur}. Suppose that $(f_1,\ldots,f_n)$ and $(\lambda_1,\ldots,\lambda_n)$ are sequences in $\real^n$. 
Let $(\lambda_i^\downarrow)_{i=1}^n$ and $(f_i^\downarrow)_{i=1}^n$ be their decreasing rearrangements. Following \cite{moa} we introduce a majorization order $(f_1,\ldots,f_n) \preceq (\lambda_1,\ldots,\lambda_n)$ if and only if
\begin{equation}\label{horn1}
\sum_{i=1}^{n}f^\downarrow_i =\sum_{i=1}^{n}\lambda^\downarrow_{i} \quad\text{ and }\quad \sum_{i=1}^{k}f^\downarrow_{i} \leq \sum_{i=1}^{k} \lambda^\downarrow_{i} \quad\text{for all } 1\le k \le n.
\end{equation}

The following result is a variant of result shown by Benac, Massey, and Stojanoff \cite{bms, bms2}. Although it was originally stated in the setting of measure spaces with almost everywhere majorization, it also holds for measurable spaces. This result also holds for real Hilbert spaces $H_n$, though the proof in \cite{bms} is shown only in the complex case. As we will see in Section \ref{S4}, Theorem \ref{obce} answers Problem \ref{psh} for type I$_n$ von Neumann algebras.

\thm{obce}{\cite[Theorem 5.1]{bms} Let $A:X\to \mathcal B(H_n)$ be a measurable field of $n\times n$ self-adjoint matrices  
with associated measurable eigenvalues $\lambda_i:X\to\real$, $i=1,\ldots,n$, such that $\lambda_1\geq\lambda_2\geq \ldots \geq\lambda_n$. 
Let $f_i:X\to\real$, $i=1,..,n$, be measurable functions. 
The following are equivalent: 
\begin{enumerate}[(i)]
\item $(f_1(x),\ldots,f_n(x)) \preceq (\lambda_1(x),\ldots,\lambda_n(x))$ for all $x\in X$.
\item There is a measurable field of unitary matrices $U: X \to \mathcal U(H_n)$ such that matrices $U^*(x)A(x)U(x)$ has the 
diagonal $(f_1(x),\ldots,f_n(x))$ for all $x\in X$.
\end{enumerate}
}

From Theorem \ref{obce} we can draw the following corollary. 

\thm{skonczone}{Let $f_i:X\to [0,1]$, $i=1,..,n$, be measurable functions and $\sum f_i(x)\in\nat$ for all $x$. 
Then, there is a measurable projection $P: X \to \mathcal B(H_n)$ with diagonal $(f_i)_{i\in \N}$.} 
\begin{proof} Sets $X_k=\{x\in X:\sum f_i(x)=k\}$ are measurable.
On every $X_k$ we have 
\[
(f_1(x),\ldots,f_n(x))\preceq (\lambda_1,\ldots,\lambda_n),\qquad\text{where }\lambda_1=\ldots=\lambda_k=1, \ \lambda_{k+1}=\ldots=\lambda_n=0.
\]
For $x\in X$, let $A_k(x)$ be a diagonal $n\times n$ matrix with diagonal $ \lambda_1(x),\ldots,\lambda_n(x)$. By Theorem \ref{obce} 
there is measurable unitary field of $n\times n$ matrices $U$  such that the matrix $U^*(x)A_k(x)U(x)$ has diagonal $(f_i(x))_{i=1}^n$. 
The mapping $x\mapsto U^*(x)A_k(x)U(x)$ is the required measurable projection.
\end{proof}

First we prove Theorem \ref{mainfin} in the following special case.

\begin{theorem}\label{mainfin2} Let $2\le N\in \N$ or $N=\infty$. Theorem \ref{mainfin} holds under the additional assumption that for all $x\in X$,
\begin{enumerate}[(i)]
\item $f_i(x) \in (0,1)$ for all $i\in \N$, 
\item
there are exactly $N$ indices $i\in\N$ such that $f_i(x) \in (0,1/2]$, and
\item
there are infinitely many $f_i(x) \in (1/2,1)$.
\end{enumerate}
\end{theorem}
 
\begin{proof}
Using Definition \ref{pos} we split the sequence $(f_i)_{i\in \N}$ into two sequences $(a_i)_{i=1}^N$ and $(b_i)_{i \in \N}$ defined as $a_i=f_{pos(i)}$ and $b_i=f_{Pos(i)}$. 
By Lemma \ref{position} functions $a_i:X\to (0,{1\over 2}],~b_i:X\to ({1\over 2},1)$ are measurable. Observe that 
\[
a(x)=\sum_{i=1}^N a_i(x), \qquad b(x)=\sum_{i=1}^\infty (1-b_i(x)).
\] 
First we will construct a projection $P$ with diagonal that coincides with $(f_i)$ except three terms $(a_{i_1},~a_{i_2},~b_{i_3})$.

For $x\in X$, define
\[
i_1(x)= \begin{cases} 1 & a_1(x) \ge a_2(x), \\
2 & \text{otherwise},
\end{cases}
\qquad
i_2(x)=3-i_1(x).
\]
Hence, $i_1,i_2: X \to \{1,2\}$ are measurable and 
 $a_{i_1}(x)\geq a_{i_2}(x)$ for all $x$. 
 
For $x\in X$, let $i_3(x)$ be the smallest $i$ with $b_i(x)\geq 1-a_{i_1}(x)$. 
The number $i_3(x)$ is well defined since the sum $b(x)$ is finite and hence $\lim_{i\to \infty} b_i(x) =1$. 
Function $i_3:X\to\nat$ is measurable since 
\[
\{i_3=k\}=\bigcap_{i=1}^{k-1} \{ b_i<1-a_{i_1} \} \cap \{ b_k\geq 1-a_{i_1}\} \qquad k\in \N.
\]
Let $i_4(x)$ be the smallest $k\in \N$, $k\ge 3$, such that 
$$ b_{i_3}(x)+ \sum_{i=k}^N a_i(x) \leq 1. $$ 
Let $i_5(x)$ be the smallest $k\in\N$ such that
$$ 1-a_{i_2}(x)+\sum_{i=k, i\ne i_3}^\infty (1-b_i(x)) \leq 1.$$ 
The existence of $i_5(x)$ follows from the convergence of the series defining $b(x)$. Note that if $N<\infty$ it might happen that $i_4(x)=N+1$, but this does not affect our construction.

\lem{i4i5}{
Functions $i_4, i_5:X\to\nat$ are measurable. }

\begin{proof}
By Lemma \ref{index} function $b_{i_3}$ is measurable. 
For any $n\ge 3$  we have
\[
\{ i_4=n\}=  \bigcap_{k=3}^{n-1} \bigg \{ b_{i_3}+\sum_{i=k}^N a_i >1\bigg\} 
\cap\bigg\{ b_{i_3}+\sum_{i=n}^N a_i \leq 1\bigg\}.
\]
Hence, $i_4$ is measurable. Likewise, for any $n\ge 1$ we have
\[
\{i_5=n\}
= \bigcap_{k=1}^{n-1} \bigg\{ 1-a_{i_2}+\sum_{i=k, i \ne i_3}^\infty (1-b_i) >1\bigg\} 
\cap \bigg\{1-a_{i_2}+\sum_{i=n,i\ne i_3}^\infty (1-b_i) \leq 1\bigg\}.
\]
Hence, $i_5$ is also measurable.
\end{proof}

With the above definitions we have $i_4\ge 3$, 
\begin{equation}\label{sumaa}  b_{i_3}+\sum_{i=i_4}^N a_i  \leq 1
\end{equation}
and 
\begin{equation}\label{sumbb}
 1-a_{i_2}+\sum_{i=i_5, ~ i\ne i_3}^\infty (1-b_i)  \leq 1. 
 \end{equation}

Define functions $\tilde{b}_{i_3}(x)$ and $\tilde{a}_{i_2}(x)$  such that 
\begin{equation}
\label{suma}  \tilde{b}_{i_3}(x)+\sum_{i=i_4}^N a_i(x) = 1
\end{equation}
and 
\begin{equation}
\label{sumb} 1-\tilde{a}_{i_2}(x)+\sum_{i=i_5,~i\ne i_3}^\infty (1-b_i(x))= 1. 
\end{equation}
Measurability of functions $\tilde{b}_{i_3}, \tilde{a}_{i_2}:X\to [0,1]$ follows from Lemma \ref{i4i5}. 
Finally, $\tilde{a}_{i_1}(x)$ is defined by the equality 
\begin{equation}\label{summ}
\tilde{a}_{i_1} + \tilde{a}_{i_2} + \tilde{b}_{i_3}~=~a_{i_1} + a_{i_2} + b_{i_3}.
\end{equation} 
Since $a_{i_1}+b_{i_3}\geq 1$ and $a_{i_2}\geq \tilde{a}_{i_2}$, we have $0\le \tilde{a}_{i_1}\le 1$. 
Hence, $\tilde{a}_{i_1}:X\to [0,1]$ is measurable as a sum/difference of measurable functions. 
By \eqref{sumaa}, \eqref{sumbb}, \eqref{suma}, and \eqref{sumb}, 
\[
\tilde a_{i_2} \le a_{i_2} \le a_{i_1}, \ \tilde a_{i_1} \le  b_{i_3} \le \tilde b_{i_3}.
\] 
Hence, by \eqref{summ} we have  
\begin{equation}\label{ts0}
 (b_{i_3}, a_{i_1}, a_{i_2})\preceq (\tilde b_{i_3}, \tilde a_{i_1}, \tilde a_{i_2}).
\end{equation} 
We also claim that 
\begin{equation}\label{ts}
\tilde{a}_{i_1}  +\quad\sum_{i=3}^{i_4-1}a_i  + \sum_{i=1, ~i\ne i_3}^{i_5-1} b_i \in \nat .
\end{equation}
Indeed, by \eqref{suma} and \eqref{sumb}
\[
\begin{aligned}
\inte&\ni\sum_{i=1}^N a_i-\sum_{i=1}^\infty (1-b_i)= 
\sum_{i=3}^N a_i-\sum_{i=1,~ i\ne i_3}(1-b_i)  +a_{i_1} +a_{i_2}-(1-b_{i_3}) 
\\
&=\sum_{i=3}^N a_i-\sum_{i=1,~ i\ne i_3}(1-b_i) +\tilde{a}_{i_1}+\tilde{a}_{i_2}+\tilde{b}_{i_3}-1  
\\
& =\sum_{i=3}^{{i_4}-1} a_i-\sum_{i=1,~ i\ne i_3}^{{i_5}-1} (1-b_i)+\tilde{a}_{i_1} +  
 \bigg(\tilde{b}_{i_3}+\sum_{i=i_4}^N a_i\bigg)  
- \bigg(1-\tilde{a}_{i_2}+\sum_{i=i_5,~ i\ne i_3}^\infty (1-b_i)\bigg) 
\\
&
= \sum_{i=3}^{{i_4}-1} a_i + \sum_{i=1,~ i\ne i_3}^{{i_5}-1} (b_i-1) + \tilde{a}_{i_1}. 
\end{aligned}
\]

We divide functions $a_i,~b_i,~\tilde{a}_{i_1},~\tilde{a}_{i_2},~\tilde{b}_{i_3}$ into three groups:
\begin{enumerate}[(I)]
\item $\{a_i: 3 \le i<i_4\}~\cup~\{b_i: 1\le i<i_5,~i\ne i_3\}~\cup~\{\tilde{a}_{i_1}\}$,
\item $\{a_i:~i\geq i_4 \}~\cup~\{\tilde{b}_{i_3}\}$, 
\item $\{b_i:~i\geq i_5,~i\ne i_3\}~\cup~\{\tilde{a}_{i_2}\}$.
\end{enumerate}
The sum of numbers in group I is a natural number by \eqref{ts}. 
The sum of numbers in both groups II and III is one by \eqref{suma} and \eqref{sumb}. 
Using the notation from the proof of Lemma \ref{ordering} we order functions in these groups in the following sequences:
\[
\begin{aligned}
(p_i^I)& =(\tilde{a}_{i_1},a_3,\ldots,a_{i_4-1},b_1,b_2,\ldots,\breve{b}_{i_3},\ldots,b_{i_5-1}),
\\
(p_i^{II}) &=(\tilde{b}_{i_3},a_{i_4}, a_{i_4+1},\ldots),
\\
(p_i^{III}) &=(\tilde{a}_{i_2},b_{i_5}, b_{i_5+1},\ldots,\breve{b}_{i_3},\ldots).
\end{aligned}
\]

Note that the sequence $(p_i^{I})$ is finite, but it has a variable length. However, we can decompose the space $X$ into measurable sets $X_n$, such that the sequence $(p_i^I)$ has length $n$ on each set $X_n$. Likewise, if $N<\infty$, then the sequence $(p_i^{II})$ has variable length as well. However, for the sake of simplicity we shall assume that $N=\infty$; hence, $(p_i^{II})$ is infinite. The case $N<\infty$ is a simple modification and it does not cause any difficulties.
 
Sequences $(p_i^I)$, $(p_i^{II})$, and $(p_i^{III})$ consist of measurable functions by the following easy lemma. 

\lem{shift}{Let $s:X\to\nat,~f_i:X\to\real$ be measurable. Then, a function at any fixed position in the sequence $f_1,\ldots,\breve{f}_s, \ldots$ is measurable. The same is true for the sequence
$f_s,f_{s+1},\ldots $} 
\begin{proof}
Restricting to the set $\{s=k\}$ for a given $k\in \N$ yields the required conclusion. 
\end{proof} 

Let $(e_i)_{i\in\N}$ be an orthonormal basis of $H$. For fixed $n\in \N$ define subspaces
\[
H^I_n = \spa\{e_1,\ldots,e_n\},
\qquad
H^{II}_n = \ov{\spa}\{e_{n+1},e_{n+3},\ldots\},
\qquad
H^{III}_n= \ov{\spa}\{e_{n+2},e_{n+4},\ldots \}.
\]
Applying Theorem \ref{skonczone} to the sequence $(p^I_i)$ of length $n$, we obtain a measurable projection $P^I_n: X_n \to \mathcal B(H^I_n)$  with diagonal $(p^I_i)_{i\in [n]}$. 
Applying Lemma \ref{sum1} we obtain measurable projections $P^{II}_n$, $P^{III}_n$ with diagonals $(p^{II}_i)_{i\in\N}$, $(p^{III}_i)_{i\in\N}$, resp.  Consequently, 
\[P_n: X_n \to \mathcal B(H), \qquad
P_n= P_n^I\oplus P_n^{II}\oplus P_n^{III}
\]
is a measurable projection with diagonal
\begin{equation}\label{pp}
(p^I_1,\ldots,p^I_n,p^{II}_1,p^{III}_1,p^{II}_2,p^{III}_2, p^{II}_3,p^{III}_3,\ldots).
\end{equation}
Define $P=\bigcup_{n\in\N} P_n: X \to \mathcal B(H)$. The sequence \eqref{pp} can be written as
\[
\tilde f_{\pi(1)},f_{\pi(2)},\ldots,f_{\pi(n)}, \tilde f_{\pi(n+1)}, \tilde f_{\pi(n+2)}, f_{\pi(n+3)}, f_{\pi(n+4)},\ldots
\]
for some measurable permutation $\pi: X \times \N \to \N$, where
\begin{equation}\label{pp2}
\tilde f_{\pi(1)} = \tilde a_{i_1}, \qquad \tilde f_{\pi(n+1)}=\tilde b_{i_3},\qquad \tilde f_{\pi(n+2)}= \tilde a_{i_2}.
\end{equation}
For each $n$ define subspaces
\[
H^{IV}_n = \spa\{e_1,e_{n+1},e_{n+2}\},
\qquad
H^{V}_n = H \ominus H^{IV}_n= \ov{\spa}\{e_{2},\ldots,e_{n}, e_{n+3},e_{n+4},\ldots\}.
\]
Define a measurable field of $3\times 3$ self-adjoint matrices $A_n: X_n \to \mathcal B(H^{IV}_n)$ by 
\[
A_n = 
\begin{bmatrix} \langle P e_1,e_1 \rangle  & \langle P e_{n+1},e_1 \rangle& \langle P e_{n+2},e_1 \rangle\\
\langle P e_1,e_{n+1} \rangle & \langle P e_{n+1},e_{n+1} \rangle & \langle P e_{n+2},e_{n+1} \rangle \\
\langle P e_1,e_{n+2} \rangle & \langle P e_{n+1},e_{n+2} \rangle & \langle P e_{n+2},e_{n+2} \rangle  
\end{bmatrix}
=
\begin{bmatrix} \tilde f_{\pi(1)} & * & * \\
* & \tilde f_{\pi(n+1)} & *\\
*& * & \tilde f_{\pi(n+2)} 
\end{bmatrix}.
\]
By \eqref{ts0} we have 
\[
(f_{\pi(n+2)}(x),f_{\pi(1)}(x) ,f_{\pi(n+1)} (x)) \preceq (\tilde f_{\pi(n+2)}(x),\tilde f_{\pi(1)}(x) ,\tilde f_{\pi(n+1)} (x)) \qquad x \in X_n .
\]
Applying Theorem \ref{obce}, there exists a measurable field of unitaries $U_n: X_n \to \mathcal B(H^{IV}_n)$ such that $U_n^*(x) A(x) U_n(x)$ has diagonal $(f_{\pi(1)}(x), f_{\pi(n+1)}(x),f_{\pi(n+2)} (x))$. 

Define a measurable projection 
\[
Q_n: X_n \to \mathcal B(H) \qquad Q_n = (U_n \oplus \mathbf I)^*P_n(U_n \oplus \mathbf I),
\]
where $\mathbf I$ is the identity on $H^V_n$. Define $Q=\bigcup_{n\in\N} Q_n: X \to \mathcal B(H)$. By our construction, $Q$ has diagonal $(f_{\pi(1)},f_{\pi(2)},\ldots)$. Since $\pi$ is a measurable permutation, Lemma \ref{permut} yields the desired measurable projection $P$. 
\end{proof}

The following result removes the last two superfluous hypotheses in Theorem \ref{mainfin2}.

\begin{theorem}\label{mainfin3}
Theorem \ref{mainfin} holds under the additional hypothesis that $f_i(x) \in (0,1)$ for all $i\in \N$ and $x\in X$.
\end{theorem}

\begin{proof}
We split $X$ into two measurable subsets \[
X'=\{x \in X: \text{ there are infinitely many $f_i(x)\in (1/2,1)$}\}
\]
and its complement $X \setminus X'$. Furthermore, $X'$ is split into measurable sets $X_N$, $N\in \N \cup \{0,\infty\}$ such that there are exactly $N$ terms $f_i(x) \in (0,1/2]$ for $x\in X_N$. We can then apply Theorem \ref{mainfin2} on each set $X_N$, where $N\ge 2$, and use Lemma \ref{sklejanie}. Hence, it suffices to show that Theorem \ref{mainfin2} also holds for $N=0,1$. However, this is a consequence of Theorem \ref{mainspec2} applied to the sequence $(1-f_i)_{i\in \N}$, which has all terms in $(0,1/2)$ with the exception of at most one term in $[1/2,1)$. Hence, there exists a measurable projection $P: X_N \to \mathcal B(H)$, $N=0, 1$, with diagonal $(1-f_i)_{i\in \N}$. Consequently, $\mathbf I - P$ is a measurable projection with diagonal $(f_i)_{i\in I}$.

Since the series defining $a(x)$ is finite, for every $x\in X$, there are only finitely many $f_i(x)=1/2$. Thus, for every $x\in X \setminus X'$, there are infinitely many $f_i(x)\in (0,1/2)$. Applying the above construction for a sequence of functions $(1-f_i)_{i\in \N}$  on $X\setminus X'$ yields a measurable projection $P: X \setminus X' \to \mathcal B(H)$ with diagonal $(1-f_i)_{i\in \N}$. Consequently, $\mathbf I - P$ is a measurable projection with diagonal $(f_i)_{i\in I}$.
\end{proof}

Finally, we shall do away with the remaining superfluous assumption in Theorem \ref{mainfin3}. We adopt the following definition.

\defi{pro}{For a sequence of measurable functions $f_i:X\to [0,1]$, $i\in \N$, we
define functions $Pro(n):X\to\nat$ and $pro(n):X\to\nat$ as follows. Let $Pro(n)(x)=k$ if
$f_k(x)$ is the $n$-th number in the sequence $(f_i(x))$ that belongs to $(0,1)$. Likewise, we let $pro(n)(x)=k$ if $f_k(x)$ is the $n$-th number in the sequence $(f_i(x))$ that is either $0$ or $1$.}

We have the following analogue of Lemma \ref{position}, which is shown in the same manner.

\lem{proper}{Functions $Pro(n)$ and $pro(n)$ are measurable for all $n$.}

\begin{proof}[Proof of Theorem \ref{mainfin}]

We split $X$ into measurable subsets corresponding to the following two cases. 

{\bf Case 1. }There are only finitely many $f_i\in (0,1)$. \newline 
For any finite sequence $k_1, \ldots,k_n$ of natural numbers the set 
$$X_{k_1,\ldots,k_n}=\bigcap_{j=1,\ldots,n} \{f_{k_j}\in (0,1)\}~\cap ~\bigcap_{i\ne k_j} (\{f_i=0\}\cup\{f_i=1\})$$ 
is a measurable subset of $X$. Let $(e_i)_{i\in \N}$ be an orthonormal basis of H. We  split the space $H$ into two orthogonal subspaces
\[
H_{k_1,\ldots,k_n}=\spa\{e_{k_i}: i=1,\ldots,n\}
\qquad
(H_{k_1,\ldots,k_n})^\bot =\ov{\spa} \{ e_j: j \ne k_i \}.
\] 
By Theorem \ref{skonczone} there is a measurable projection $P_{k_1, \ldots,k_n}: X_{k_1, \ldots, k_n} \to \mathcal B(H_{k_1,\ldots,k_n})$ with diagonal $(f_{k_i})_{i\in [n]}$. On the space $(H_{k_1,\ldots,k_n})^\bot$ 
there is an obvious diagonal projection $Q_{k_1,\ldots,k_n}$ with zeros and ones on the diagonal. 
The projection $P_{k_1,\ldots,k_n}\oplus Q_{k_1,\ldots,k_n}$ is a sought projection on $X_{k_1,..,k_n}$ acting on the whole space $H$. 
Since there are countably many sets $X_{k_1,\ldots,k_n}$, and they are disjoint for different $k_1,\ldots,k_n$, we 
are done on a measurable set 
\[
\bigcup_{n=0}^\infty \bigcup_{k_1,\ldots,k_n \in \N} X_{k_1,\ldots,k_n}.
\]
The case $n=0$ corresponds to all zeros and ones in $(f_i)$.

{\bf Case 2.} There are infinitely many $f_i\in (0,1)$. \newline
We split $X$ into measurable subsets
\[
X_N = \{ x\in X: \text{there are exactly $N$ values $f_i(x)=0$ or $1$}\}, \qquad N \in \N \cup \{0,\infty\}.
\]
That way functions $pro(n)$ are defined on $X_N$ for all $1\le n\le N$. For simplicity we shall assume that $N=\infty$. The case $N<\infty$ is a simple modification and it does not cause any difficulty.

Define a measurable projection $\pi: X_\infty \times \N \to \N$ by
\[
\pi(n)= \begin{cases}
Pro(n/2) & n \text{ is even},
\\
pro((n+1)/2) & n \text{ is odd}.
\end{cases}
\]
Define orthogonal subspaces
\[
H_0 = \ov{\spa}\{e_{2n}: n\in \N\},
\qquad
H_1= \ov{\spa}\{e_{2n-1}: n\in \N \}.
\]
By Thoerem \ref{mainfin3}, there exists a measurable projection $P_0: X_\infty \to \mathcal B(H_0)$ with diagonal $(f_{\pi(2n)})_{n\in\N}$. Let $P_1: X_\infty \to \mathcal B(H_1)$ be the obvious diagonal projection with zeros and ones on the diagonal $(f_{\pi(2n-1)})_{n\in\N}$. Then, $P_0 \oplus P_1$ is a measurable projection with diagonal $(f_{\pi(n)})_{n\in\N}$. Applying Lemma \ref{permut} yields the desired measurable projection.
\end{proof}

\section{Applications}\label{S4}

In this section we present applications of Theorem \ref{sane} to shift-invariant spaces and von Neumann algebras.

\subsection{Shift-invariant spaces}
Shift-invariant (SI) spaces are closed subspaces of $L^2(\R^d)$ that are
invariant under all shifts, i.e., integer translations. That is, a closed subspace $V \subset  L^2(\R^d)$ is SI if $T_k(V)=V$ for all $k\in\Z^d$,
where $T_kf(x)=f(x-k)$ is the translation operator.
The theory of
shift-invariant spaces plays an important role in many
areas, most notably in the theory of wavelets,  spline systems, Gabor systems,
and approximation theory \cite{BDR1, BDR2, b, RS1, RS2}.
The study of analogous spaces for $L^2(\mathbb T, H)$ with values in a
separable Hilbert space $H$ in terms of the range function, often called
doubly-invariant spaces, is quite classical and goes back to Helson \cite{h}.

In the context of SI spaces a {\it range function} is any mapping
\begin{equation}\label{range}
J: \mathbb T^d \to \{Y \subset \ell^2(\Z^d): Y \text{ is a closed subspace}\},
\end{equation}
where $\mathbb T^d=\R^d/\Z^d$ is
identified with its fundamental domain $[-1/2,1/2)^d$. We say that $J$ is {\it
measurable} if the associated orthogonal projections $P_J(\xi)$ of $\ell^2(\Z^d)$ onto $J(\xi)$ are weakly operator measurable in the sense of Definition \ref{mp}. We follow the convention which identifies range functions if they are equal a.e. 
A fundamental result due to Helson \cite[Theorem 8, p.~59]{h} gives one-to-one correspondence between SI spaces $V$ and measurable range functions $J$, see also \cite[Proposition 1.5]{b}. This is achieved using a fiberization operator $\mathcal T : L^2(\R^d) \to L^2(\mathbb T^d, \ell^2(\mathbb Z^d))$ given by
\[
\mathcal T f(\xi) = (\hat f(\xi+k))_{k\in \Z^d} \qquad\text{for } f\in L^2(\R^d),\  \xi \in \T^d,
\]
where $\hat f(\xi)= \int_{\R^d} f(x) e^{-2\pi i\langle x, \xi \rangle} dx$ is the Fourier transform of $f\in L^1(\R^d)$, and extended unitarily to $L^2(\R^d)$ by the Plancherel theorem. Then, the one-to-one correspondence between SI spaces $V \subset L^2(\R^n)$ and measurable range functions $J$ is encapsulated by the formula
\[
V = \{ f\in L^2(\R^d): \mathcal T f(\xi) \in J(\xi) \quad\text{for a.e. }\xi\in \R^d\}.
\]

Spectral function of SI spaces were introduced by Rzeszotnik and the first author in \cite{br, br2}, see also \cite{dut, gh}. While there are several equivalent ways of introducing the spectral function of a SI space, the most relevant definition uses a range function.

\begin{definition}\label{spec}
The {\it spectral function} of a SI space $V$ is a measurable mapping 
$\sigma_V: \R^d \to [0,1]$ given by
\begin{equation}\label{dsp}
\sigma_V(\xi+k) = ||P_J(\xi) e_k||^2=\langle P_J(\xi)e_k,e_k \rangle \qquad\text{for } \xi\in \mathbb T^d,\ k\in\Z^d,
\end{equation}
where $(e_k)_{k\in\Z^d}$ denotes the standard basis of $\ell^2(\Z^d)$ and $\mathbb T^d=[-1/2,1/2)^d$. In other words, $(\sigma_V(\xi+k))_{k\in \Z^d}$ is a diagonal of a projection $P_J(\xi)$.
\end{definition}

Note that $\sigma_V(\xi)$ is well defined for a.e.~$\xi\in\R^d$,
since $\{ k+ \mathbb T^d: k\in \Z^d\}$
is a partition of $\R^d$. Theorem \ref{sane} yields the following characterization of spectral functions.

\begin{theorem}\label{sp}
Let $\sigma: \R^d \to [0,1]$ be a measurable function. For $\xi\in \R^d$ define
\[
a(\xi)=\sum_{k\in \Z^d, \ \sigma(\xi+k)<1/2} \sigma(\xi+k) \quad\text{and}\quad 
b(\xi)=\sum_{k\in \Z^d, \ \sigma(\xi+k) \ge1/2}(1-\sigma(\xi+k)).
\]
The following are equivalent:
\begin{enumerate}[(i)]
\item
There exists a SI space $V \subset L^2(\R^d)$ such that its spectral function $\sigma_V$ coincides almost everywhere with $\sigma$,
\[
\sigma_V(\xi) = \sigma(\xi) \qquad\text{for a.e. }\xi \in \R^d.
\]
\item
For a.e. $\xi \in \R^d$ we either have
\begin{itemize}
\item $a(\xi),b(\xi)<\infty$ and  $a(\xi)-b(\xi)\in\Z$, or
\item $a(\xi)=\infty$ or $b(\xi)=\infty$.
\end{itemize}
\end{enumerate}
\end{theorem}

\begin{proof}
The implication $(i)\implies (ii)$ follows from the necessity part of Kadison Theorem \ref{Kadison}. Indeed, by Definition \ref{spec} the sequence $(\sigma_V(\xi+k))_{k\in\Z^d}$ is a diagonal of a projection $P_J(\xi)$ acting on $\ell^2(\Z^d)$ for a.e. $\xi \in\T^d$. Hence, (ii)  holds pointwise a.e.

The converse implication $(ii) \implies (i)$ is the crux of this result.  Kadison's Theorem \ref{Kadison} implies merely the existence of projections $P(\xi)$ with desired diagonal $(\sigma(\xi+k))_{k\in\Z^d}$ for a.e. $\xi$. However, it does not guarantee the measurability of $\xi \mapsto P(\xi)$. This is where we need to apply Theorem \ref{sane} instead. Let $X$ be the set of all points $\xi\in \R^d$ such that (ii) holds. This is a set of full measure, i.e., $\R^d \setminus X$ is a null set. Theorem \ref{sane} yields the existence of a measurable projection $P: X \to \mathcal B(\ell^2(\Z^d))$ such that $P(\xi)$ has diagonal $(\sigma(\xi+k))_{k\in\Z^d}$ for a.e. $\xi$. This projection correspond to a measurable range function $J$ as in \eqref{range}, which by Helson's theorem corresponds to a SI space $V \subset L^2(\R^d)$.
\end{proof}

The notion of a spectral function can be extended to the setting of translation invariant (TI) subspaces of $L^2(G)$ of a second countable locally countable group $G$, see \cite{br3}. In this case, $\Gamma \subset G$ is a closed co-compact subgroup and the spectral function represents diagonal entries of a measurable projection defined on $X=\hat G/\Gamma^*$ with values in  $\mathcal B(\ell^2(\Gamma^*))$, where $\Gamma^*$ is the annihilator of $\Gamma$ in the dual group $\hat G$. Since Theorem \ref{sane} does not put any restriction on a measurable space $X$, an extension of Theorem \ref{sp} to this setting does not cause any extra difficulties.

\subsection{Carpenter's Theorem for type I$_\infty$ von Neumann algebras}
In this subsection we will answer Problem \ref{psh} for von Neumann algebras of type I$_\infty$ when $T$ is a projection. 

The structure theorem for type I von Neumann algebras \cite[Section I.3.12]{black} gives the following characterization. A type I$_n$, where $n\in \N \cup \{\infty\}$, von Neumann algebra $\mathfrak M$ is isomorphic with the tensor product $\mathfrak M \cong \mathcal B(H_n) \otimes L^\infty(X,\mu)$, where $H_n$ is $n$ dimensional Hilbert space and $(X,\mu)$ is locally finite measure space. Hence, $\mathfrak M$ is isomorphic to $L^\infty(X,\mu,\mathcal B(H_n))$ acting on a Hilbert space $L^2(X,\mu,H_n)$. 

More precisely, let $H$ be the direct integral of a measurable field of infinite dimensional separable Hilbert spaces $H_\infty$ on $(X,\mu)$ given by
\begin{equation}\label{di1}
H = \int^\oplus_X H_\infty d\mu(x).
\end{equation}
Let $\mathfrak M$ be the subalgebra of $\mathcal B(H)$ consisting of all decomposable operators
\begin{equation}\label{di2}
T = \int^\oplus_X T(x) d\mu(x),
\end{equation}
where $x\mapsto T(x) \in \mathcal B(H_\infty)$ is an essentially bounded measurable field of operators, see \cite[Section 2.3]{Dix}. The structure theorem for type I$_\infty$ von Neumann algebra implies that every such algebra is of this form. Consider MASA $\mathcal A \subset \mathcal M$ consisting of all decomposable operators $T$ such that for $\mu$-a.e. $x\in X$, $T(x)$ is a diagonal operator with respect to some fixed orthonormal basis $(e_i)_{i=1}^\infty$ of $H_\infty$. Then, Theorem \ref{sane} implies the following theorem.

\begin{theorem}\label{vn}
Let $\mathfrak M$ be the von Neumann algebra of type {\rm I}$_\infty$ consisting of decomposable operators \eqref{di2} acting on the direct integral Hilbert space \eqref{di1}. Let $\mathcal A\subset \mathfrak M$ be a MASA consisting of decomposable operators whose fibers $T(x)$ are diagonal operators with respect to some fixed orthonormal basis of $H_\infty$.

Suppose that $P$ is a projection in $\mathfrak M$. Define the dimension functions $p,q: X \to \N \cup \{0,\infty\}$ by
\[
p(x) = \operatorname{rank} (P(x)), 
\qquad
q(x) = \operatorname{rank} (\mathbf I - P(x)) \qquad\text{for }x\in X.
\]
Then, the set $D_{\mathcal A}(P)$ of conditional expectations of the unitary orbit of $P$, which is given by \eqref{di0}, consists of operators $T$ such that $T(x)$ is a diagonal operator on $\mathcal H_\infty$ with entries given by a sequence of measurable functions $f_i: X \to [0,1]$, $i\in \N$, satisfying the following conditions for $\mu$-a.e. $x\in X$:
\begin{enumerate}[(i)]
\item
$\sum_{i\in \N} f_i(x)= p(x)$, $\sum_{i\in \N} (1-f_i(x))=q(x)$, and
\item functions $a,b:X \to [0,\infty]$ given by \eqref{sane0} satisfy either: 
\begin{itemize}
\item
$a(x)=\infty$ or $b(x)=\infty$, or
\item
$a(x), b(x)<\infty$ and $a(x)-b(x)\in\Z$.
\end{itemize}
\end{enumerate}
\end{theorem}

\begin{proof}
Let $U \in \mathfrak M$ be a unitary operator. Since $U$ is a decomposable operator its fibers $U(x)$ are unitary operators for  a.e. $x$. Thus, $U^*PU$ is a decomposable operator with fibers $U(x)^*P(x)U(x)$. Let $(f_i(x))_{i=1}^\infty$ be the diagonal of $U(x)^*P(x)U(x)$. Then, by the trace argument and by the necessity part in Kadison's Theorem \ref{Kadison}, the diagonal sequence satisfies (i) and (ii), respectively.

The converse implication is a consequence of Theorem \ref{sane}. Let $X' \subset X$ be the subset of full measure for which either (i) or (ii) holds for all $x \in X'$. By Theorem \ref{sane}, there exists a measurable projection $Q: X' \to \mathcal B(H_\infty)$ with diagonal $(f_i)_{i\in \N}$. We extend $Q$ to $X$ in any way. Then,
\begin{equation*}\label{di4}
Q = \int^\oplus_X Q(x) d\mu(x)
\end{equation*}
is a projection in $\mathfrak M$ with diagonal $(f_i)_{i\in\N}$ modulo null sets. It remains to show that there exists a unitary $U \in \mathfrak M$ such that $Q=U^*PU$.

Measurable projections $P$ and $Q$ correspond to measurable range functions 
\[
J_P, J_Q: X \to \{ Y \subset H_\infty: Y \text{ is a closed subspace}\}.
\]
Let $P^\perp=\mathbf I -P$ and $Q^\perp = \mathbf I - Q$ be the projections on orthogonal subspaces, which correspond to measurable range functions 
\[
J_{P^\perp}(x) = (J_P(x))^\perp,
\qquad
J_{Q^\perp}(x) = (J_Q(x))^\perp.
\]
By (i) we have 
\[
\dim J_P(x)=\dim J_Q(x)=p(x)
\quad\text{and}\quad
\dim J_{P^\perp}(x)=\dim J_{Q^\perp}(x)=q(x)
\qquad\text{for $\mu$-a.e. }x.
\]
By Helson's theorem \cite[Theorem 2 in Section 1.3]{h2}, there exists a sequence of measurable functions $G_i: X \to H_\infty$, $i\in\N$, such that for $\mu$-a.e. $x\in X$, $\{G_i(x)\}_{i=1}^{p(x)}$ forms an orthonormal basis of $J_P(x)$. Let $\tilde G_i$, $i\in\N$, be the corresponding orthonormal basis sequence for the range function $J_Q$. Likewise, there exists measurable functions $F_i: X \to H$ and $\tilde F_i: X \to H_\infty$, $i\in\N$, such that $\{F_i(x)\}_{i=1}^{q(x)}$ and $\{\tilde F_i(x)\}_{i=1}^{q(x)}$ are an orthonormal basis of $J_{P^\perp}(x)$ and $J_{Q^\perp}(x)$, resp. Let $U(x)$ be the unitary operator on $H_\infty$ which maps the orthonormal basis $\{\tilde G_i(x)\}_{i=1}^{p(x)} \cup \{\tilde F_i(x)\}_{i=1}^{q(x)}$ onto $\{G_i(x)\}_{i=1}^{p(x)} \cup \{F_i(x)\}_{i=1}^{q(x)}$. Then, our construction yields $Q(x) = U(x)^* P(x) U(x)$. Consequently,
\begin{equation*}
U = \int^\oplus_X U(x) d\mu(x)
\end{equation*}
is the required unitary satisfying $Q=U^*PU$, which completes the proof.
\end{proof}

We conjecture that an analogue of Theorem \ref{vn} holds for self-adjoint operators $T\in \mathfrak M$ such that $T(x)$ has a finite spectrum for $\mu$-a.e. $x\in X$. The necessary conditions are provided by the corresponding result for I$_{\infty}$ factors, that is $\mathcal B(H_\infty)$, which was shown by Jasper and the first author \cite[Theorem 1.3]{bj2}. However, the sufficiency requires a construction of a measurable field of unitary operators and, a priori, it is not clear if this is possible. The lack of any obstruction for operators with two point spectrum, which are essentially projections, suggests an affirmative answer to this problem as well.

\end{document}